\documentclass[11pt, reqno]{amsproc} 

\usepackage{geometry}
\usepackage{color}
\usepackage{verbatim}
\usepackage[utf8x]{inputenc}
\usepackage{t1enc}
\usepackage[english]{babel}

\usepackage{amsmath,amssymb,amsthm}
\usepackage{mathrsfs,esint}
\usepackage{graphicx,wrapfig}
\usepackage[format=hang]{caption}

\newcommand*{\R}{{\mathbb R}}

\newcommand*{\N}{{\mathbb N}}

\newcommand*{\eps}{\varepsilon}

\providecommand*{\vint}[1]{\mathchoice
          {\mathop{\vrule width 5pt height 3 pt depth -2.5pt
                  \kern -9pt \kern 1pt\intop}\nolimits_{\kern -5pt{#1}}}
          {\mathop{\vrule width 5pt height 3 pt depth -2.6pt
                  \kern -6pt \intop}\nolimits_{\kern -3pt{#1}}}
          {\mathop{\vrule width 5pt height 3 pt depth -2.6pt
                  \kern -6pt \intop}\nolimits_{\kern -3pt{#1}}}
          {\mathop{\vrule width 5pt height 3 pt depth -2.6pt
                  \kern -6pt \intop}\nolimits_{\kern -3pt{#1}}}}
\newcommand*{\jint}{\fint}

\DeclareMathOperator{\Mod}{Mod}

\numberwithin{equation}{section}
\theoremstyle{plain}
\newtheorem{thm}[equation]{Theorem}

\newtheorem{lem}[equation]{Lemma}

\theoremstyle{definition}

\newtheorem{remark}[equation]{Remark}

\begin{document}

\title[Homogeneous Newton-Sobolev spaces]{Homogeneous Newton-Sobolev spaces in metric measure spaces and their Banach space properties.} 
\author{Nageswari Shanmugalingam}
\address{Department of Mathematical Sciences, P.O.~Box 210025, University of Cincinnati, Cincinnati, OH~45221-0025, U.S.A.}
\email{shanmun@uc.edu}
\thanks{
 N.S.'s work is partially supported by the NSF (U.S.A.) grant DMS~\#2054960.}
\maketitle

\begin{abstract}
In this note we prove the Banach space properties of the homogeneous Newton-Sobolev spaces 
$HN^{1,p}(X)$
of functions on an unbounded
metric measure space $X$ equipped with a doubling measure supporting a $p$-Poincar\'e inequality, and show
that when $1<p<\infty$, 
even with the lack of global $L^p$-integrability of functions in $HN^{1,p}(X)$, we have that every bounded sequence in
$HN^{1,p}(X)$ has a strongly convergent convex-combination subsequence.
The analogous properties for the inhomogeneous Newton-Sobolev classes $N^{1,p}(X)$ are proven
elsewhere in existing literature (as for example in~\cite{Sh1, HKSTbook}).
\end{abstract}


\medskip

\noindent
    {\small \emph{Mathematics Subject Classification (2020):}
Primary: 46E36.
Secondary: 31E05, 31L15.
}

\section{Introduction}

The Sobolev spaces play a vital role in the study of partial differential equations and potential theory in 
the setting of Euclidean spaces, and more generally, Riemannian manifolds. The research carried out during
the past two decades also developed function spaces, in the setting of more general metric measure spaces,
that can play the role of Sobolev spaces in the setting where the metric space is non-smooth and carries
no natural differential structure. One such class of function-spaces, called the Newton-Sobolev spaces,
have turned out to be highly useful in developing potential theory and studying properties of energy-minimizers.
The homogeneous versions of these spaces is the focus of the present paper.

Homogeneous versions of Newton-Sobolev spaces, where the energy of the function in that version is globally
finite without the function itself being necessarily in the global $L^p$-class, arise naturally when dealing with 
large-scale potential theory of a metric measure space; see for example~\cite{F, Holo, HoloKos, N, BCS}.
For this reason, the behavior of such functions has attracted great interest recently, and an interesting discussion
about limits at infinity of such functions can be found in the metric space context in~\cite{KNW, KN, EKN, KKN}.

As with classical Sobolev spaces, the Banach space property is vital in the use of Newton-Sobolev spaces.
The normed vector-space property of these function spaces is clear.
It was shown in~\cite{Sh1} (see also~\cite{HKSTbook}) that the Newton-Sobolev spaces $N^{1,p}(X)$, with
$1\le p<\infty$ and $(X,d,\mu)$ the non-smooth metric measure space on which functions from $N^{1,p}(X)$
act, are always Banach spaces, that is, as normed vector spaces, $N^{1,p}(X)$ are complete. The situation of
homogeneous Newton-Sobolev spaces $D^{1,p}(X)$ is different; even the norm on $D^{1,p}(X)$ is only a
seminorm. In this note we will establish the 
Banach space property for $HN^{1,p}(X):=D^{1,p}(X)/\sim$, where $\sim$ is an equivalence relationship on $D^{1,p}(X)$ that
ensures that the seminorm on $D^{1,p}(X)$ transforms into a norm on $HN^{1,p}(X)$.

\begin{thm}\label{thm:main}
If the measure $\mu$ of the metric measure space $(X,d,\mu)$ is doubling and supports a $p$-Poincar\'e inequality,
then $HN^{1,p}(X)$ is a Banach space. If $1<p<\infty$, then any bounded sequence in $HN^{1,p}(X)$ has a convex-combination
subsequence that converges in $HN^{1,p}(X)$.
\end{thm}

Here by a convex combination subsequence of a sequence $(u_k)_k$ we mean a sequence of the form
$(w_k)_k$ such that for each $k\in\N$ there are nonnegative real numbers $\lambda_{j,k}$, $j=1,\cdots, N_k$
with $\sum_{j=0}^{N_k}\lambda_{j,k}=1$ and  $w_k=\sum_{j=0}^{N_k}\lambda_{j,k}\, u_{k+j}$. Should $HN^{1,p}(X)$ be 
reflexive, then by Mazur's lemma the second claim of the theorem would follow from the first part of the theorem; 
the proof that $HN^{1,p}(X)$ is reflexive follows along the same lines as the proof the fact that $N^{1,p}(X)$ is reflexive
under the hypotheses of the above theorem, with added complications due to lack of global $L^p$-integrability of the functions
in $HN^{1,p}(X)$; see for example the discussion in~\cite{HKSTbook}, which uses the tools developed in~\cite{Chee}. 
However, the consequence of Mazur's lemma does not need the reflexivity property of $HN^{1,p}(X)$, and we provide a
more direct proof here.

\section{Background}

The classical Sobolev spaces are based on the notion of weak derivatives, where we expect to have a vector-valued
function associated with a Sobolev function so that an integration-by-parts formula holds. 
The comprehensive text~\cite{Maz} gives a wide-ranging overview of the theory of Sobolev spaces in Euclidean setting, while
the lecture note~\cite{Vaisala}
gives a connection between classical Sobolev functions and functions that are absolutely continuous on almost every
curve in the Euclidean space, and so for almost every curve, the function should be differentiable almost everywhere
along that curve. This is the idea behind the notion of upper gradients, first proposed by Heinonen and Koskela in~\cite{HK},
see also~\cite{Hei}. 

In this note, $(X,d,\mu)$ is a metric measure space with $\mu$ a Radon measure such that for each ball $B\subset X$ we have
that $0<\mu(B)<\infty$. We say that $\mu$ is \emph{doubling} if there is a constant $C\ge 1$ such that
\[
\mu(B(x,2r))\le C\, \mu(B(x,r)),
\]
whenever $x\in X$ and $r>0$. Here $B(x,r):=\{y\in X\, :\, d(x,y)<r\}$ denotes the ball centered at $x$ and with radius $r$.
We will also fix $1\le p<\infty$ here.

\subsection{Construction of Newton-Sobolev spaces $N^{1,p}(X)$}
A function $f:X\to [-\infty,\infty]$ is said to have a Borel function $g:X\to[0,\infty]$ as an \emph{upper gradient} if we have
\[
|f(\gamma(b))-f(\gamma(a))|\le \int_\gamma g\, ds
\]
whenever $\gamma:[a,b]\to X$ is a rectifiable curve with $a<b$. In the above, we interpret the inequality as also
meaning that $\int_\gamma g\, ds=\infty$ whenever at least one of $f(\gamma(b))$, $f(\gamma(a))$ is not finite.
We say that $f\in \widetilde{N^{1,p}}(X)$ for a fixed choice of $1\le p<\infty$ if 
\[
\Vert f\Vert_{N^{1,p}(X)}:=\left(\int_X|f|^p\, d\mu\right)^{1/p}+\inf_g\left(\int_Xg^p\, d\mu\right)^{1/p}
\]
is finite. In the above, the infimum is over all upper gradients $g$ of $f$. 
It is not difficult to see that $\widetilde{N^{1,p}}(X)$ is a vector space, and that if $g_1$ is an upper gradient of $f_1$ and
$g_2$ is an upper gradient of $f_2$, then $|\alpha|g_1+|\beta|g_2$ is an upper gradient of $\alpha f_1+\beta f_2$ when
$\alpha,\beta\in\R$.

Note that it is possible for a function
to be not identically zero on $X$ but at the same time have $\Vert f\Vert_{N^{1,p}(X)}=0$. This is seen, for
example, in the context of Euclicean spaces, with $f=\chi_{\{0\}}$, see the discussion in~\cite{HKSTbook}. Therefore,
to set $\Vert \cdot\Vert_{N^{1,p}(X)}$ to be a norm, we deal with a quotient space. To this end, when $f_1,f_2\in \widetilde{N^{1,p}}(X)$,
we say that $f_1\sim f_2$ if $\Vert f_1-f_2\Vert_{N^{1,p}(X)}=0$.

We say that a family $\Gamma$ of rectifiable curves in $X$ has \emph{$p$-modulus zero}, that is, $\Mod_p(\Gamma)=0$, if
there is a non-negative Borel function $\rho$ on $X$ such that $\int_X\rho^p\, d\mu$ is finite, and at the same time,
$\int_\gamma\rho\, ds=\infty$ for each $\gamma\in \Gamma$. It follows from~\cite{HKSTbook} that the countable union of
families of curves with zero $p$-modulus must have $p$-modulus zero.

With $E\subset X$, the collection $\Gamma_E$ is the family of all non-constant compact rectifiable curves in $X$ that intersect $E$.
The collection $\Gamma_E^+$ is the family of all non-constant compact rectifiable curves $\gamma$ in $X$ for which 
$\mathcal{H}^1(\gamma_s^{-1}(E))>0$, where $\gamma_s:[0,\ell(\gamma)]\to X$ is the arc-length parametrization of $\gamma$.
Here, $\ell(\gamma)$ is the length of the rectifiable curve $\gamma$. Note that when $E$ is a Borel set with $\mu(E)=0$, necessarily
we must have that $\Mod_p(\Gamma_E^+)=0$. However, it is still possible to have $\Mod_p(\Gamma_E)$ to be positive.

Following Ohtsuka, we say that a set $E\subset X$ is \emph{$p$-exceptional} if both $\mu(E)=0$ and $\Mod_p(\Gamma_E)=0$
are satisfied.

\begin{lem}\label{lem:p-Except}
If $f_1, f_2\in \widetilde{N^{1,p}}(X)$, then $f_1\sim f_2$ if and only if the set $E:=\{x\in X\, :\, f_1(x)\ne f_2(x)\}$ is $p$-exceptional.
\end{lem}

\begin{proof}
Let $f_1,f_2\in \widetilde{N^{1,p}}(X)$. If $E:=\{x\in X\, :\, f_1(x)\ne f_2(x)\}$ satisfies $\mu(E)=0$, then we already 
have that $\int_X|f_1-f_2|^p\, d\mu=0$. If in addition we have that $\Mod_p(\Gamma_E)=0$, then with $\rho$ the 
Borel function associated with this property for $\Gamma_E$, we see that for each $\eps>0$ the function
$g_\eps:=\eps \rho$ is an upper gradient of $f_1-f_2$, and $\int_Xg_\eps^p\, d\mu=\eps^p\, \int_X\rho^p\, d\mu$. As
$\int_X\rho^p\, d\mu<\infty$, it follows that $\inf_g\int_Xg^p\, d\mu=0$ where the infimum is over all upper gradients $g$ of 
$f_1-f_2$.

Now suppose that $f_1\sim f_2$. Then we must have that $\int_X|f_1-f_2|^p\, d\mu=0$, and so $\mu(E)=0$. 
Moreover, for each positive integer $n$ we can find an upper gradient $g_n$ of $f_1-f_2$ such that $\int_Xg_n^p\, d\mu<2^{-np}$. Set
$\rho_0=\sum_ng_n$. For each positive integer $k$ let $E_k=\{x\in X\, :\, |f_1(x)-f_2(x)|>1/k\}$. Then $E=\bigcup_kE_k$.

By the summability of $\sum_n2^{-n}$, we know that $\int_X\rho_0^p\, d\mu<\infty$. Since $\mu$ is
a Radon measure, there is a Borel set $E_0\subset X$ with $E\subset E_0$ such that $\mu(E_0)=0$; let $\rho=\rho_0+\infty\chi_{E_0}$.
Then again we have that $\int_X\rho^p\, d\mu=\int_X\rho_0^p\, d\mu<\infty$, and moreover, if $\gamma\in\Gamma_E^+$,
then $\int_\gamma\rho\, ds=\infty$. On the other hand, if $\gamma\in\Gamma_{E_k}\setminus\Gamma_E^+$, then there is some
point along the trajectory of $\gamma$ at which $f_1-f_2$ takes on the value of $0$, and there is some point in this trajectory where
$|f_1-f_2|$ is at least $1/k$. It follows that for each positive integer $n$ we have $\int_\gamma g_n\ge 1/k$, and so for each
such $\gamma$ we have that $\int_\gamma k\rho\, ds=\infty$. Thus $\Mod_p(\Gamma_{E_k}\cup\Gamma_E^+)=0$.
As the countable union of families of curves, each of zero $p$-modulus, must be of $p$-modulus zero, it follows that
$\Gamma_E$ is of zero $p$-modulus. This completes the proof of the lemma.
\end{proof}

The \emph{Newton-Sobolev class} is the collection of equivalence classes, that is, $N^{1,p}(X):=\widetilde{N^{1,p}}(X)/\sim$. 
Given the above lemma, we see that when $f_1\sim f_2$,
we have $\Vert f_1\Vert_{N^{1,p}(X)}=\Vert f_2\Vert_{N^{1,p}(X)}$, and so the seminorm $\Vert\cdot\Vert_{N^{1,p}(X)}$ is inherited by
$N^{1,p}(X)$ as a norm. Indeed, with $E$ the set of points $x\in X$ for which $f_1(x)\ne f_2(x)$, we have that $E$ is $p$-exceptional,
and so there is some non-negative Borel function $\rho\in L^p(X)$ such that $\int_\gamma\rho\, ds=\infty$ for each 
$\gamma$ that intersects $E$; hence, for each $\eps>0$ we have that $\eps\rho$ is an upper gradient of $f_1-f_2$, and 
so whenever $g_1, g_2$ are upper gradients of $f_1, f_2$ respectively, then $g_i+\eps\rho$ is an upper gradient
of $f_j$ for $i, j\in\{1,2\}$ with $i\ne j$. Thus $\Vert f_1\Vert_{N^{1,p}(X)}=\Vert f_2\Vert_{N^{1,p}(X)}$.

\begin{lem}[{\cite[Theorem~3.7]{Sh1},\cite[Proposition~7.3.1, Proposition~7.2.8]{HKSTbook}}]\label{lem:ptwise-Banach}
The space $N^{1,p}(X)$ is a Banach space.
Suppose that $(u_k)_k$ is a Cauchy sequence of functions in $N^{1,p}(X)$, converging to $u\in N^{1,p}(X)$. 
Let $(u_{k_j})_j$ be any subsequence of this sequence satisfying $\Vert u_{k_j}-u_{k_{j+1}}\Vert\le C\, 2^{-j}$ with $C>0$ that is
independent of $j$. Then, with $u_{k_j}$ also denoting a representative
function from the equivalence class $u_{k_j}\in N^{1,p}(X)$, and setting $E$ to be the collection of all points $x\in X$ for which
$(u_{k_j}(x))_j$ is not a Cauchy sequence of real numbers, we have that $E$ is a $p$-exceptional set.
\end{lem}

\subsection{Construction of homogeneous Newton-Sobolev spaces $HN^{1,p}(X)$}

We set $D^{1,p}(X)$ to be the collection of all functions $f:X\to[-\infty,\infty]$ that have an upper gradient $g$ such that
$\int_Xg^p\, d\mu$ is finite, that is, $g$ is $p$-summable. Note that as constant functions have the identically-zero function
as an upper gradient, it follows that constant functions belong to $D^{1,p}(X)$, regardless of whether $\mu(X)$ is finite or infinite.
For $f\in D^{1,p}(X)$ we set
\[
\Vert f\Vert_{D^{1,p}(X)}=\inf_g\left(\int_X g^p\, d\mu\right)^{1/p},
\]
where the infimum is over all upper gradients $g$ of $f$. Note that $\Vert\cdot\Vert_{D^{1,p}(X)}$ is only a seminorm on $D^{1,p}(X)$,
as non-zero constant functions have their seminorm be zero. As in the previous subsection, we say that two functions $f_1,f_2\in D^{1,p}(X)$
are equivalent, $f_1\sim f_2$, if $\Vert f_1-f_2\Vert_{D^{1,p}(X)}=0$. 

\begin{lem}
Let $f_1,f_2\in D^{1,p}(X)$ such that $f_1\sim f_2$. Then there is a family $\Gamma$ of non-constant compact rectifiable curves in $X$
with $\Mod_p(\Gamma)=0$, such that whenever $\gamma$ is a non-constant compact rectifiable curves in $X$ with $\gamma\not\in\Gamma$,
we have that $f_1\circ\gamma=f_2\circ\gamma$. Moreover, $\Vert f_1\Vert_{D^{1,p}(X)}=\Vert f_2\Vert_{D^{1,p}(X)}$.
\end{lem}

\begin{proof}
Since $f_1\sim f_2$, for each $n\in\N$ we can find an upper gradient $g_n$ of $f_1-f_2$ such that $\int_Xg_n^p\, d\mu<2^{-np}$.
As before, let $\rho=\sum_ng_n$, and note that $\rho\in L^p(X)$. Let $\Gamma_b$ be the collection of all non-constant compact
rectifiable curves $\gamma$ in $X$ for which $\int_\gamma\rho\, ds=\infty$. 
It follows that 
$\Mod_p(\Gamma_b)=0$. 

 For ease of notation, we set $u:=f_1-f_2$.
Fix any $\gamma\not\in\Gamma_b$, say, $\gamma:[a,b]\to X$. Note that
as $\int_\gamma\rho\, ds<\infty$, it follows that for each $s,t\in[a,b]$ with $s<t$ we also have that $\gamma\vert_{[s,t]}\not\in\Gamma_b$.
Hence for each $n\in\N$ we have that
\begin{equation}\label{eq:A}
|u(\gamma(s))-u(\gamma(t))|\le \int_{\gamma\vert_{[s,t]}}g_n\, ds\le \int_\gamma g_n\, ds.
\end{equation}
A theorem of Fuglede~\cite{HaK, HKSTbook} yields that there is a family $\Gamma_F$ of non-constant compact rectifiable curves
with $\Mod_p(\Gamma_F)=0$ such that whenever $\gamma\not\in\Gamma_F$ is a non-constant compact rectifiable curve in $X$, we have
that 
\[
\lim_{n\to\infty}\int_\gamma g_n\, ds=0.
\]
Note that with $\Gamma=\Gamma_b\cup\Gamma_F$, we have $\Mod_p(\Gamma)=0$. So if $\gamma\not\in\Gamma$, then
by~\eqref{eq:A} we have that for each $s,t\in[a,b]$ with $s<t$,
\[
|u(\gamma(s))-u(\gamma(t))|\le 0.
\]
Thus $u\circ\gamma$ is constant for each $\gamma\not\in\Gamma$.

The last claim of the lemma follows from the triangle inequality for the seminorms: 
\[
\Vert f_1\Vert_{D^{1,p}(X)}\le \Vert f_2\Vert_{D^{1,p}(X)}+\Vert f_2-f_1\Vert_{D^{1,p}(X)}=\Vert f_2\Vert_{D^{1,p}(X)},
\]
and
\[
\Vert f_2\Vert_{D^{1,p}(X)}\le \Vert f_1\Vert_{D^{1,p}(X)}+\Vert f_2-f_1\Vert_{D^{1,p}(X)}=\Vert f_1\Vert_{D^{1,p}(X)}.
\]

\end{proof}

The \emph{homogeneous Newton-Sobolev space} $HN^{1,p}(X)$ is the collection of all equivalence classes of $D^{1,p}(X)$
under the equivalence relationship $\sim$, and we denote $HN^{1,p}(X):=D^{1,p}(X)/\sim$. Thanks to the above lemma, this
homogeneous space inherits the seminorm of $D^{1,p}(X)$ as a norm.
Note that in general, it is not true that $f_1\sim f_2$ implies that $f_1-f_2$ is a constant; for example, if $X$ has no non-constant
compact rectifiable curves, then all functions on $X$ are equivalent as zero function would be an upper gradient for it. This is in
stark contrast to $\widetilde{N^{1,p}}(X)$, where if $f_1$ is equivalent to $f_2$ with equivalence as defined in the previous subsection,
then $f_1=f_2$ almost everywhere. There are some circumstances under which $D^{1,p}(X)$ has such a regulated behavior, and
one such circumstance is that in which $X$ supports a $p$-Poincar\'e inequality.

\begin{remark}
We point out here that in some literature the term \emph{homogeneous Sobolev space} in relation to Euclidean spaces stands
in for the $HN^{1,p}$--norm-closure of the collection of all compactly supported smooth functions on Euclidean spaces,
see for instance~\cite{KY}. We do not consider this analog here, but it would be interesting to know what the relationship
between these two spaces are in general. Our definition of homogeneous spaces seem to be more intimately connected to
trace theorems for unbounded domains~\cite{GKS}.
\end{remark}

\subsection{Poincar\'e inequalities}

As in the previous subsections, we fix an index $p$ with $1\le p<\infty$. 
For measurable sets $A\subset X$ with $0<\mu(A)<\infty$  and $f:X\to\R$ a measurable function, we set
\[
\jint_Af\, d\mu:=\frac{1}{\mu(A)}\int_Af\, d\mu=:f_A.
\]

We say that the metric measure space $(X,d,\mu)$
supports a $p$-Poincar\'e inequality if there are constants $C>0$ and $\lambda\ge 1$ such that, whenever
$f$ is a measurable function on $X$ and $g$ is an upper gradient of $f$, and $B(x,r)$ is a ball in $X$, we have
\[
\jint_{B(x,r)}|f-f_B|\, d\mu\le C\, r\, \left(\jint_{B(x,\lambda r)}g^p\, d\mu\right)^{1/p}.
\]

Since we assume that balls in $X$ have positive and finite $\mu$-measure, it follows that $X$ is connected whenever
$(X,d,\mu)$ supports a $p$-Poincar\'e inequality. Under some additional assumptions (such as doubling property of $\mu$
and local compactness of $X$) it also follows that $X$ is quaisconvex, that is, there is a constant $C_q\ge 1$ such that whenever
$x,y\in X$ are two distinct points, there is a rectifiable curve $\gamma_{x,y}$ with end points $x$ and $y$ and with
length $\ell(\gamma_{x,y})\le C_q\, d(x,y)$, see for instance~\cite[Theorem~8.3.2]{HKSTbook} or~\cite{Chee, Kor}. 
In this subsection we will not assume that $\mu$ is doubling, and so the quasiconvexity
property is not necessarily true.

\begin{lem}
Suppose that $(X,d,\mu)$ supports a $p$-Poincar\'e inequality and $f_1,f_2\in D^{1,p}(X)$ such that $f_1\sim f_2$. Then there
exists $C=C_{f_1-f_2}\in\R$ such that $f_2=f_1+C$ $\mu$-a.e.~in $X$.
\end{lem}

\begin{proof}
Suppose that $f_1\sim f_2$. Set $u:=f_1-f_2$. Then as $\Vert u\Vert_{D^{1,p}(X)}=0$, for each $n\in\N$ we can find an upper 
gradient $g_n$ of $u$ such that $\int_Xg_n^p\, d\mu\le 2^{-np}$.

We fix a point $x_0\in X$ and $R>0$ and denote $B=B(x_0,R)$. Then by the Poincar\'e inequality, we have that 
\[
\jint_B|u-u_B|\, d\mu\le C\, R\, \left(\jint_{B(x_0,\lambda R)}g_n^p\, d\mu\right)^{1/p}\le \frac{C\, R}{\mu(B(x_0,\lambda R))^{1/p}}\, 2^{-n}.
\]
Letting $n\to\infty$ we see that we must have $u=u_B$ $\mu$-a.e.~in $B$. 

Now, exhausting $X$ by replacing $B$ by $B(x_0,kR)$ for each $k\in\N$ and applying the above argument, we now conclude that
$u$ is to be constant $\mu$-a.e.~in $X$; choosing $C$ to be that constant completes the proof of the lemma.
\end{proof}

The following theorem was proved in~\cite{HaK}.

\begin{thm}[{\cite[Theorem~5.1]{HaK}}]\label{thm:p-p}
If the measure $\mu$ is doubling and supports a $p$-Poincar\'e inequality, then there are constants $C_0>0$ and $\lambda_0\ge 1$ such that
whenever $f:X\to[-\infty,\infty]$ is a measurable function with upper gradient $g\in L^p(X)$ and $B(x,r)$ is a ball in $X$, then we have
\[
\left(\jint_{B(x,r)}|f-f_B|^p\, d\mu\right)^{1/p}\le C_0\, r\, \left(\jint_{B(x,\lambda_0 r)}g^p\, d\mu\right)^{1/p}.
\]
\end{thm}

\section{Banach space property of $HN^{1,p}(X)$}

The goal of this section is to provide a proof of Theorem~\ref{thm:main}. We start with the following useful tool.

\begin{lem}\label{lem:useful1}
Suppose the hypotheses of Theorem~\ref{thm:main} hold.
Let $u\in HN^{1,p}(X)$ with upper gradient $g$, and let $x_0\in X$. If $0<R_0<R<\infty$, then setting $B_0=B(x_0,R_0)$, we have 
\[
\left(\jint_{B(x_0,R)}|u-u_{B_0}|^p\, d\mu\right)^{1/p}
\le C_0\, \left(1+\left(\frac{\mu(B(x_0,R))}{\mu(B_0)}\right)^{1/p}\right)\, R\, \left(\jint_{B(x_0,\lambda R)}g^p\, d\mu\right)^{1/p}.
\]
\end{lem}

\begin{proof}
We have that 
\[
\left(\jint_{B(x_0,R)}|u-u_{B_0}|^p\, d\mu\right)^{1/p}\le \left(\jint_{B(x_0,R)}|u-u_{B(x_0,R)}|^p\, d\mu\right)^{1/p}+|u_{B(x_0,R)}-u_{B_0}|.
\]
Note that
\begin{align*}
|u_{B(x_0,R)}-u_{B_0}|
 \le \jint_{B_0}|u-u_{B(x_0,R)}|\, d\mu
 \le & \left(\jint_{B_0}|u-u_{B(x_0,R)}|^p\, d\mu\right)^{1/p}\\
 \le & \left(\frac{\mu(B(x_0,R))}{\mu(B_0)}\jint_{B(x_0,R)}|u-u_{B(x_0,R)}|^p\, d\mu\right)^{1/p}.
\end{align*}
Now an application of Theorem~\ref{thm:p-p} yields the claim of the lemma.
\end{proof}

Now we are ready to prove Theorem~\ref{thm:main}.

\begin{proof}[Proof of Theorem~\ref{thm:main}]
Let $(u_k)_k$ be a Cauchy sequence in $HN^{1,p}(X)$, and we fix a ball $B_0=B(x_0,R_0)$. With $u_k$ also denoting any choice of
representative function from the equivalence class $u_k\in HN^{1,p}(X)$ for each $k\in\N$, we set the function $v_k$ by
\[
v_k:=u_k-(u_k)_{B_0},
\]
where $(u_k)_{B_0}=\jint_{B_0}u_k\, d\mu$ is the mean of $u_k$ on the ball $B_0$. Note that the mean of $v_k$ on the same ball $B_0$
is zero, i.e., $(v_k)_{B_0}=0$. Note that for each $k\in\N$, $v_k$ belongs to the same equivalence class as $u_k$, that is, as far as
$HN^{1,p}(X)$ is concerned, $v_k=u_k$. 

As $(v_k)_k$ is Cauchy in $HN^{1,p}(X)$, we know that for each $m,n\in\N$ with $m>n$, there is an upper gradient $g_{v_m-v_n}$ of
$v_m-v_n$ such that $\lim_{m,n\to\infty}\int_Xg_{v_m-v_n}^p\, d\mu=0$. Therefore, by passing to a subsequence if necessary, we may
assume that for each $k\in\N$ we have $\int_Xg_{v_{k+1}-v_k}^p\, d\mu<2^{-kp}$. Note that if we can show the convergence of this subsequence
in $HN^{1,p}(X)$, then the entire Cauchy sequence also converges.

We fix $R>R_0$, and consider the sequence of \emph{functions} $v_k\in L^p(B(x_0,R))$. By an application of Lemma~\ref{lem:useful1},
we know that 
\begin{align*}
\int_{B(x_0,R)}|v_{k+1}-v_k|^p\, d\mu&\le C_0^p\, \frac{R^p}{\mu(B(x_0,R))}\, \left(1+\left(\frac{\mu(B(x_0,R))}{\mu(B_0)}\right)^{1/p}\right)^p\, \int_Xg_{v_{k+1}-v_k}^p\, d\mu\\
  &\le C_0^p\, \frac{R^p}{\mu(B(x_0,R))}\, \left(1+\left(\frac{\mu(B(x_0,R))}{\mu(B_0)}\right)^{1/p}\right)^p\, 2^{-kp}.
\end{align*}
Setting 
\[
A(R):=C_0\, \frac{R}{\mu(B(x_0,R))^{1/p}}\, \left(1+\left(\frac{\mu(B(x_0,R))}{\mu(B_0)}\right)^{1/p}\right),
\]
we now have that
\[
\Vert v_{k+1}-v_{k}\Vert_{L^p(B(x_0,R))}\le A(R)\, 2^{-k}.
\]
It follows that the sequence $(v_k)_k$ is Cauchy in $L^p(B(x_0,R))$, and hence converges to a function $v_R$ there. Moreover, by the 
above inequality, we also know that this convergence is also \emph{pointwise $\mu$-a.e.~in} the ball $B(x_0,R)$.
Since this happens for each $R>R_0$, it follows that there is a function $v\in L^p_{loc}(X)$ with $v_k\to v$ both in
$L^p_{loc}(X)$ and $\mu$.a.e.~in $X$.

It only remains to show that (a) there is a function $\widehat{v}\in D^{1,p}(X)$ with $\widehat{v}=v$ $\mu$-a.e.~in $X$,
and that (b) $v_k\to \widehat{v}$ in $HN^{1,p}(X)$. To do so, note that $(v_k\vert_{B(x_0,R)})_k$ is a Cauchy 
sequence in $N^{1,p}(B(x_0,R))$, thanks to Lemma~\ref{lem:useful1} and the condition on the choices $g_{u_m-u_n}$
which are upper gradients of $v_m-v_n$ on $B(x_0,R)$ as well. Indeed, this sequence also satisfies the second hypothesis
of Lemma~\ref{lem:ptwise-Banach} with the entire sequence being considered as its own subsequence. Thus, by
Lemma~\ref{lem:ptwise-Banach}, we know that there is a function $w_R\in N^{1,p}(B(x_0,R))$ such that
$v_k\to w_R$ in $N^{1,p}(B(x_0,R))$. As the $L^p(B(x_0,R))$-norm is part of the $N^{1,p}(B(x_0,R))$-norm,
it follows that $w_R=v$ $\mu$-a.e.~in $B(x_0,R)$. Moreover, this lemma also tells us that $v_k\to w_R$ pointwise in
$B(x_0,R)\setminus E[R]$ for some $p$-exceptional set $E[R]$. Here we take $E[R]$ to be the set of points $x$ in
$B(x_0,R)$ for which $(v_k(x))_k$ is not a Cauchy sequence; thus, we also have that $E[R_2]\cap B(x_0,R_1)=E[R_1]$
whenever $R_2>R_1>R_0$. Hence, when $R_0<R_1<R_2$, we also have that $w_{R_2}=w_{R_1}$ on $B(x_0,R_1)$,
and so we can find a globally defined function $\widehat{v}$ so that
the restriction of $\widehat{v}$ to $B(x_0,R)$ is $w_R$.

For each $k\in\N$, note that 
$\sum_{j=k}^\infty g_{v_{j+1}-v_j}$ is an upper gradient of $v_l-v_k$ in $X$
for each $l\in\N$ with $l>k$. It follows that $\sum_{j=k}^\infty g_{v_{j+1}-v_j}$ is an upper gradient
of $\widehat{v}-v_k$ in $B(x_0,R)$ for each $R>R_0$. 
It now follows from the fact that $v_k\in D^{1,p}(X)$ that $\widehat{v}\in D^{1,p}(X)$, that is, $\widehat{v}$ belongs
to an equivalence class in $HN^{1,p}(X)$.

As
\[
\Vert \sum_{j=k}^\infty g_{v_{j+1}-v_j}\Vert_{L^p(X)}\le \sum_{j=k}^\infty \Vert g_{v_{j+1}-v_j}\Vert_{L^p(X)}
<2^{1-k},
\]
it follows that 
\[
\Vert \widehat{v}-v_k\Vert_{D^{1,p}(X)}<2^{1-k},
\]
that is, $v_k\to \widehat{v}$ in $D^{1,p}(X)$ and hence in $HN^{1,p}(X)$. This completes the proof of the first part of the theorem.

To show prove the last claim of the theorem when $1<p<\infty$, we need to show that bounded sequences in there have a convex combination
subsequence that converges to some function in $HN^{1,p}(X)$. To this end, let $(v_k)_k$ be a representative sequence from 
a bounded sequence in $HN^{1,p}(X)$. We can, by replacing $v_k$ with $v_k-(v_k)_{B_0}$ if necessary, assume that
each $v_k$ has mean $(v_k)_{B_0}=0$. Then, from Lemma~\ref{lem:useful1} and by the $p$-Poincar\'e inequality, we know that
for each $R_0<R<\infty$ the sequence $(v_k)_k$ is bounded in $N^{1,p}(B(x_0,R))$. By~\cite[Theorem~7.3.12]{HKSTbook},
which uses only the reflexivity property of the spaces $L^p$, we know that $N^{1,p}(B(x_0,R))$ satisfies
the conclusion of Mazur's lemma. Therefore
it follows that there is a convex-combination subsequence $w_{R,j}=\sum_{k=j}^{N_j}\, \lambda_{j,k}\, v_k$ with $0\le \lambda_{j,k}\le 1$
for each $k=j,\cdots, N_j$ and $\sum_{k=j}^{N_j}\lambda_{j,k}=1$, and a function $w_R\in N^{1,p}(B(x_0,R))$, such that
$w_{R,j}\to w_R$ in $N^{1,p}(B(x_0,R))$. 

Now choosing a strictly monotone increasing sequence of radii $R_k$ with $\lim_kR_k=\infty$, and for each $k\ge 2$ making sure
to consider the sequence $(w_{R_{k-1},j})_j$ on $B(x_0,R_j)$, and utilizing a Cantor-type diagonalization argument, we obtain
the function $w$. Note that for each $k,j$, the function $w_{R_{k-1},j}$ has average $0$ on the ball $B_0$.
As in the previous paragraph, we see that $w_R$ can be chosen to
be independent of $R$, and so we obtain a consistent globally defined function $w$. This completes the proof.
\end{proof}

\end{document}